\documentclass[12pt]{amsart}
\usepackage{url}

\usepackage[top=1.25in, bottom=1.25in, left=1.05in, right=1.05in]{geometry}

\usepackage{amsthm,amsmath,amssymb,graphicx}
\usepackage[colorlinks=true,citecolor=black,linkcolor=black,urlcolor=blue]{hyperref}

\theoremstyle{plain}
\newtheorem{definition}{Definition}[section]
\newtheorem{remark}[definition]{Remark}

\newtheorem{proposition}[definition]{Proposition}
\newtheorem{lemma}[definition]{Lemma}
\newtheorem{theorem}[definition]{Theorem}
\newtheorem{corollary}[definition]{Corollary}

\numberwithin{equation}{section}
\begin{document}
\title[On Hurwitz type poly-Bernoulli numbers and polynomials]{Recurrences for values of the Hurwitz type poly-Bernoulli numbers and polynomials}
\author{Mohamed Amine Boutiche}
\address[M.A. Boutiche]{Faculty of Mathematics\\ USTHB \\ P.O. Box 32 El Alia 16111\\ Algiers \\ Algeria.}
\email{mboutiche@usthb.dz}
\author{Mohamed Mechacha}
\address[M. Mechacha]{Faculty of Mathematics\\ USTHB \\ P.O. Box 32 El Alia 16111\\ Algiers \\ Algeria.}
\email{mmechacha@usthb.dz}
\author{Mourad Rahmani}
\address[M. Rahmani]{Faculty of Mathematics\\ USTHB \\ P.O. Box 32 El Alia 16111\\ Algiers \\ Algeria.}
\email{rahmani.mourad@gmail.com, mrahmani@usthb.dz}

\begin{abstract}
The main object of this paper is to investigate a new class of the generalized Hurwitz type poly-Bernoulli numbers and polynomials from which we derive some algorithms for evaluating the Hurwitz type poly-Bernoulli numbers and polynomials. By introducing a new generalization of the Stirling numbers of the second kind, we succeed to establish some combinatorial formulas for the generalized Hurwitz type poly-Bernoulli numbers and polynomials with negative upper indices. Moreover, we give a connection between the generalized Stirling numbers of the second kind and graph theory.
\end{abstract}

\maketitle

\section{Introduction}
An interesting extension of the well-known Riemann zeta function is the Hurwitz-Lerch zeta function $\Phi\left(  z,s,a\right)  $ defined by \cite{Srivastava2}
\begin{gather*}
\Phi\left(  z,s,a\right)  =%
{\displaystyle\sum\limits_{n\geq0}}
\frac{z^{n}}{\left(  n+a\right)  ^{s}},\\
(s\in\mathbb{C}\text{ when }\left\vert z\right\vert <1;\operatorname{Re}%
\left(  s\right)  >1\text{ when }\left\vert z\right\vert =1),
\end{gather*}
where $a\in \mathbb{C}-\left\{  0,-1,-2,\ldots\right\}$.

Some important special cases of the Hurwitz-Lerch zeta function are Hurwitz zeta function $\zeta(s,a)=\Phi\left(  1,s,a\right)$, polylogarithm functions $\operatorname{Li}_{s}\left(  z\right)  =z\Phi\left(  z,s,1\right)$ and Dirichlet eta function $\eta(s)=\Phi\left(  -1,s,1\right)$.

The Hurwitz type poly-Bernoulli numbers $\mathcal{HB}_{n}^{\left(  k\right)  }\left(
a\right)  $ was introduced by Cenkci and Young in a recent paper \cite{Cencki}
as a generalization of poly-Bernoulli numbers, which are defined by the
following generating function
\[
\Phi\left(  1-e^{-z},k,a\right)  ={\displaystyle\sum\limits_{n\geq0}}
\mathcal{HB}_{n}^{\left(  k\right)  }\left(  a\right)  \frac{z^{n}}{n!}.
\]
The poly-Bernoulli numbers $\mathcal{B}_{n}^{\left(  k\right)  },$ given by
\[
\mathcal{B}_{n}^{\left(  k\right)  }:=\mathcal{HB}_{n}^{\left(  k\right)  }\left(  1\right)
\]
are defined by the following generating function:
\[
\frac{\operatorname*{Li}_{k}\left(  1-e^{-z}\right)  }{1-e^{-z}}=%
{\displaystyle\sum\limits_{n\geq0}}
\mathcal{B}_{n}^{\left(  k\right)  }\frac{z^{n}}{n!}.
\]
The numbers
\[
B_{n}:=\mathcal{B}_{n}^{\left(  1\right)  }\left(  1\right)
\]
are the ordinary Bernoulli numbers with $B_{1}=1/2$. For more details
on these numbers, we refer the reader to \cite{Kaneko1,Kaneko2}.

An explicit formula for $\mathcal{HB}_{n}^{\left(  k\right)  }\left(  a\right)  $ is
given by \cite{Cencki}
\[
\mathcal{HB}_{n}^{\left(  k\right)  }\left(  a\right)  =%
{\displaystyle\sum\limits_{i=0}^{n}}
\frac{\left(  -1\right)  ^{n+i}i!S\left(  n,i\right)  }{\left(  i+a\right)
^{k}},
\]
where $S\left(  n,i\right)  $ are the Stirling numbers of the second kind
arising as coefficients in the following expansion:%
\[
x^{n}={\sum\limits_{i=0}^{n}}i!\dbinom{x}{i}S\left(  n,i\right)  .
\]
The Hurwitz type poly-Bernoulli polynomials $\mathcal{HB}_{n}^{\left(  k\right)  }\left(
x;a\right)  $ is defined by the following generating function%
\[
\Phi\left(  1-e^{-z},k,a\right)  e^{-xz}=%
{\displaystyle\sum\limits_{n\geq0}}
\mathcal{HB}_{n}^{\left(  k\right)  }\left(  x;a\right)  \frac{z^{n}}{n!}.
\]
The coefficients in Cauchy's product series are given by
\[
\mathcal{HB}_{n}^{\left(  k\right)  }\left(  x;a\right)  =%
{\displaystyle\sum\limits_{i=0}^{n}}
\left(  -1\right)  ^{n-i}\dbinom{n}{i}\mathcal{HB}_{i}^{\left(  k\right)  }\left(
a\right)  x^{n-i}.
\]
In this paper, we propose to investigate a new class of the generalized
Hurwitz type poly-Bernoulli polynomials $\mathbb{B}_{n,m}^{\left(
k\right) }\left(  x;a\right)  $ which we call $m-$Hurwitz type
poly-Bernoulli polynomials. We establish several properties of these
polynomials. As a consequence, the study of $\mathbb{B}_{n,m}^{\left(
k\right) }\left( x;a\right)  $ yields an interesting algorithm for
calculating $\mathcal{HB}_{n}^{\left( k\right)  }\left(  x;a\right)  .$ The
idea is to construct an infinite matrix $\left(
\mathbb{B}_{n,m}^{\left(  k\right)  }\left(  x;a\right)  \right)
_{n,m\geq0}$, the first column of which gives the Hurwitz
type poly-Bernoulli polynomials $\mathbb{B}_{n,0}^{\left(  k\right)
}\left( x;a\right)  :=\mathcal{HB}_{n}^{\left(  k\right)  }\left(  x;a\right)
.$ Furthermore, we introduce a new generalization of the Stirling
numbers of the second kind, which aid us to prove the explicit
formulas of the $m-$Hurwitz type poly-Bernoulli numbers and
polynomials with negative upper indices.

We first recall some basic definitions and some results \cite{Comtet,Srivastava2} that will be useful in
the rest of the paper. For $\nu\in\mathbb{C},$ the Pochhammer symbol
$(\nu)_{n}$ is defined by
\[
(\nu)_{n}=\nu\left(  \nu+1\right)  \cdots\left(  \nu+n-1\right)
\qquad\text{and}\qquad(\nu)_{0}=1.
\]
The (signed) Stirling numbers $s\left(  n,i\right)  $ of the first kind are
the coefficients in the following expansion:
\[
x\left(  x-1\right)  \cdots\left(  x-n+1\right)  ={\sum\limits_{i=0}^{n}%
}s\left(  n,i\right)  x^{i}.
\]
and satisfy the recurrence relation given by
\begin{equation}
s\left(  n+1,i\right)  =s\left(  n,i-1\right)  -ns\left(  n,i\right)
\qquad\left(  1\leq i\leq n\right)  . \label{recst}%
\end{equation}
The exponential generating functions for $s(n,i)$ and $S(n,i)$ are given by
\[
\frac
{1}{i!}\left[  \ln\left(  1+z\right)  \right]  ^{i}={\sum\limits_{n=i}^{\infty}}s\left(  n,i\right)  \;\frac{z^{n}}{n!}%
\]
and
\[
\frac
{1}{i!}\left(  e^{z}-1\right)  ^{i}={\sum\limits_{n=i}^{\infty}}S\left(  n,i\right)  \;\frac{z^{n}}{n!},
\]
respectively.

The weighted Stirling numbers $\mathcal{S}_{n}^{i}\left(  x\right)  $
of the second kind are defined by (see \cite{Carlitz1,Carlitz2})
\begin{align}
\mathcal{S}_{n}^{i}\left(  x\right)   &  =\frac{1}{i!}\Delta^{i}x^{n}\nonumber \\
&  =\frac{1}{i!}{\sum\limits_{j=0}^{i}}\left(  -1\right)  ^{i-j}\dbinom{i}%
{j}\left(  x+j\right)  ^{n}, \label{ist}
\end{align}
where $\Delta$ denotes the forward difference operator. The exponential
generating function of $\mathcal{S}_{n}^{k}\left(  x\right)  $ is given by
\begin{equation}
\frac{1}{i!}e^{xz}\left(  e^{z}-1\right)  ^{i}={\sum\limits_{n=i}^{\infty}}\mathcal{S}_{n}^{i}\left(  x\right)\;\frac{z^{n}%
}{n!} \label{wst}%
\end{equation}
and weighted Stirling numbers $\mathcal{S}_{n}^{i}\left(  x\right)  $ satisfy
the following recurrence relation:
\[
\mathcal{S}_{n+1}^{i}\left(  x\right)  =\mathcal{S}_{n}^{i-1}\left(  x\right)
+\left(  x+i\right)  \mathcal{S}_{n}^{i}\left(  x\right)  \qquad(1\leq i\leq
n).
\]
In particular, we have for non-negative integer $r$
\[
\mathcal{S}_{n}^{i}\left(  0\right)  =S\left(  n,i\right)\;\; \text{and}\;\; \mathcal{S}_{n}^{i}\left(  r\right)  =%
\genfrac{\{}{\}}{0pt}{0}{n+r}{i+r}_{r},
\]
where $%
\genfrac{\{}{\}}{0pt}{0}{n}{i}_{r}$ denotes the $r$-Stirling numbers of the second kind \cite{Broder}.

\section{The $m$-Hurwitz type poly-Bernoulli numbers}
For every non-negative integer $m,$ we define a sequence of rational numbers
$\mathbb{B}_{n,m}^{\left(  k\right)  }\left(  a\right)  $ which we call
$m$-Hurwitz type poly-Bernoulli numbers, by%
\begin{equation}
\mathbb{B}_{n,m}^{\left(  k\right)  }\left(  a\right)  =\frac{\left(
m+a\right)  ^{k}}{m!a^{k}}\sum_{i=0}^{n}%
\genfrac{\{}{\}}{0pt}{0}{n+m}{i+m}%
_{m}\frac{\left(  -1\right)  ^{n-i}\left(  i+m\right)  !}{\left(
i+m+a\right)  ^{k}}, \label{Form1}%
\end{equation}
with, of course $\mathbb{B}_{n,0}^{\left(  k\right)  }\left(  a\right)
=\mathcal{HB}_{n}^{\left(  k\right)  }\left(  a\right)  $ and $\mathbb{B}_{0,m}^{\left(
k\right)  }\left(  a\right)  =\frac{1}{a^{k}}$.

The next explicit formula for $m$-Hurwitz type poly-Bernoulli numbers can be
derived from a known result in \cite[p. 681, Corollary 1]{Rahmani2014} for the
Stirling transform.

\begin{proposition}
The $m$-Hurwitz type poly-Bernoulli numbers may be expressed in the form%
\begin{equation}
\mathbb{B}_{n,m}^{\left(  k\right)  }\left(  a\right)  =\frac{\left(
m+a\right)  ^{k}}{m!a^{k}}%
{\displaystyle\sum\limits_{i=0}^{m}}
\left(  -1\right)  ^{m-i}s\left(  m,i\right)\mathcal{HB}_{n+i}^{\left(  k\right)
}\left(  a\right)  . \label{Form2}%
\end{equation}

\end{proposition}

The following theorem contains the Rodrigues-type formula for the exponential generating function of
$m$-Hurwitz type poly-Bernoulli numbers.

\begin{theorem}
The exponential generating function for $m$-Hurwitz type poly-Bernoulli
numbers is given by
\begin{equation}
\frac{1}{m!}e^{-mz}\left(  1+\frac{m}{a}\right)  ^{k}\left(  e^{z}\frac{d}{dz}\right)  ^{m}\left[ \left(  1-e^{-z}\right)
^{m}\Phi\left(  1-e^{-z},k,m+a\right) \right] ={\sum\limits_{n\geq0}}\mathbb{B}_{n,m}^{(k)}\left(  a\right)  \frac{z^{n}}{n!} .\label{Generating1}
\end{equation}
\end{theorem}

\begin{proof}
It follows from (\ref{Form1}) and (\ref{wst}) that
\begin{align*}%
{\displaystyle\sum\limits_{n\geq0}}
\mathbb{B}_{n,m}^{(k)}\left(  a\right)  \frac{z^{n}}{n!}  &  =\frac{\left(
-1\right)  ^{m}\left(  m+a\right)  ^{k}}{m!a^{k}}\sum_{i\geq0}\frac{\left(
-1\right)  ^{m+i}\left(  i+m\right)  !}{\left(  i+m+a\right)  ^{k}}%
{\displaystyle\sum\limits_{n\geq i}}
\genfrac{\{}{\}}{0pt}{0}{n+m}{i+m}%
_{m}\frac{\left(  -z\right)  ^{n}}{n!}\\
&  =\frac{\left(  -1\right)  ^{m}\left(  m+a\right)  ^{k}}{m!a^{k}}\sum
_{i\geq0}\frac{\left(  -1\right)  ^{m+i}\left(  i+m\right)  !}{\left(
i+m+a\right)  ^{k}}\frac{1}{i!}e^{-mz}\left(  e^{-z}-1\right)  ^{i}\\
&  =e^{-mz}\left(  1+\frac{m}{a}\right)  ^{k}\sum_{i\geq0}\binom{m+i}{i}%
\frac{\left(  1-e^{-z}\right)  ^{i}}{\left(  i+m+a\right)  ^{k}}.
\end{align*}
Since
\[
\dbinom{m+i}{i}\left(  1-e^{-z}\right)  ^{i}=\frac{1}{m!}\left(  e^{z}\frac
{d}{dz}\right)  ^{m}\left(  1-e^{-z}\right)  ^{m+i},
\]
we get
\[%
{\displaystyle\sum\limits_{n\geq0}}
\mathbb{B}_{n,m}^{(k)}\left(  a\right)  \frac{z^{n}}{n!}=\frac{1}{m!}e^{-mz}\left(  1+\frac{m}{a}\right)  ^{k}\left(  e^{z}\frac{d}{dz}\right)  ^{m}\left[ \left(  1-e^{-z}\right)
^{m}\sum_{i\geq0}\frac{\left(
1-e^{-z}\right)  ^{i}}{\left(  i+m+a\right)  ^{k}}\right],
\]
which is obviously equivalent to (\ref{Generating1}). We have thus completed
the proof of the theorem.
\end{proof}

\begin{theorem}
\label{TTT}The $\mathbb{B}_{n,m}^{(k)}\left(  a\right)  $ satisfies the
following three-term recurrence relation:%
\begin{equation}
\mathbb{B}_{n+1,m}^{(k)}\left(  a\right)  =\frac{\left(  m+1\right)  \left(
m+a\right)  ^{k}}{\left(  m+a+1\right)  ^{k}}\mathbb{B}_{n,m+1}^{(k)}\left(
a\right)  -m\mathbb{B}_{n,m}^{(k)}\left(  a\right)  , \label{Mach}%
\end{equation}
with the initial sequence given by
\[
\mathbb{B}_{0,m}^{\left(  k\right)  }\left(  a\right)  =\frac{1}{a^{k}}.
\]

\end{theorem}

\begin{proof}
From (\ref{recst}) and (\ref{Form2}), we have
\[
\mathbb{B}_{n,m+1}^{(k)}\left(  a\right)=\frac{\left(  m+a+1\right)  ^{k}}{\left(  m+1\right)  !a^{k}}\
{\displaystyle\sum\limits_{i=0}^{m+1}}
\left(  -1\right)  ^{m+1-i}\left(  \left(  s\left(  m,i-1\right)  -ms\left(
m,i\right)  \right)  \right)  \mathcal{HB}_{n+i}^{\left(  k\right)  }\left(  a\right).
\]
After some rearrangement, we find that%
\[
\mathbb{B}_{n,m+1}^{(k)}\left(  a\right)  =\frac{\left(  m+a+1\right)  ^{k}%
}{\left(  m+1\right)  \left(  m+a\right)  ^{k}}\left(  \mathbb{B}%
_{n+1,m}^{(k)}\left(  a\right)  +m\mathbb{B}_{n,m}^{(k)}\left(  a\right)
\right).
\]
This evidently completes the proof of the theorem.
\end{proof}

\begin{remark}
By setting $k=1$ and $a=1$ in (\ref{Mach}), we get%
\[
\mathbb{B}_{0,m}=1,\;\;\mathbb{B}_{n+1,m}=\frac{\left(  m+1\right)  ^{2}}{\left(
m+2\right)  }\mathbb{B}_{n,m+1}\left(  a\right)  -m\mathbb{B}_{n,m}\left(
a\right)  ,
\]
which coincides with Rahmani's algorithm for Bernoulli numbers \cite{Rahmani3}
with $B_{1}=1/2.$
\end{remark}

\section{The $m$-Stirling numbers of the second kind}

In this section, we introduce a new generalization of the familiar Stirling
numbers $S(n,k)$ of the second kind, which we call $m-$Stirling numbers of the
second kind. We then derive several elementary properties.

\begin{definition}
The $m$-Stirling numbers $\mathcal{R}_{n}^{k}\left(  m\right)  $ of the second
kind is defined by
\begin{equation}
\mathcal{R}_{n}^{k}\left(  m\right)  =\frac{m!}{k!}%
{\displaystyle\sum\limits_{j=0}^{k}}
\left(  -1\right)  ^{k-j}\dbinom{k}{j}\dbinom{j+m-1}{m}j^{n}. \label{F1}%
\end{equation}

\end{definition}

Since
\[
\left(  x\right)  _{m}=m!\dbinom{x+m-1}{m},
\]
we can write $\mathcal{R}_{n}^{k}\left(  m\right)  $ as%
\begin{equation}
\mathcal{R}_{n}^{k}\left(  m\right)  =\frac{1}{k!}%
{\displaystyle\sum\limits_{j=0}^{k}}
\left(  -1\right)  ^{k-j}\dbinom{k}{j}j^{n}\left(  j\right)  _{m}. \label{Amn}
\end{equation}
Substituting $m=0$ into above equation, we have the Stirling numbers of the
second kind%
\[
\mathcal{R}_{n}^{k}\left(  0\right)  =S\left(  n,k\right)  .
\]

\begin{theorem}
The generating function of $\mathcal{R}_{n}^{k}\left(  m\right)  $ are given
by%
\begin{equation}
e^{z}\left(  e^{-z}\frac{d}{dz}\right)  ^{m}\left(  \frac{1}{k!}e^{\left(
m-1\right)  z}\left(  e^{z}-1\right)  ^{k}\right)  =%
{\displaystyle\sum\limits_{n\geq k}}
\mathcal{R}_{n}^{k}\left(  m\right)  \frac{z^{n}}{n!}. \label{F2}%
\end{equation}

\end{theorem}

\begin{proof}
By using (\ref{F1}), we obtain%
\begin{align*}%
{\displaystyle\sum\limits_{n\geq k}}
\mathcal{R}_{n}^{k}\left(  m\right)  \frac{z^{n}}{n!}  &  =\frac{m!}{k!}%
{\displaystyle\sum\limits_{j=0}^{k}}
\left(  -1\right)  ^{k-j}\dbinom{k}{j}\dbinom{j+m-1}{m}%
{\displaystyle\sum\limits_{n\geq0}}
j^{n}\frac{z^{n}}{n!}\\
&  =\frac{m!}{k!}%
{\displaystyle\sum\limits_{j=0}^{k}}
\dbinom{j+m-1}{m}e^{jz}\left[  \left(  -1\right)  ^{k-j}\dbinom{k}{j}\right]
\\
&  =\frac{m!}{k!}e^{z}%
{\displaystyle\sum\limits_{j=0}^{k}}
\dbinom{m+j-1}{j-1}e^{\left(  j-1\right)  z}\left(  -1\right)  ^{k-j}%
\dbinom{k}{j}.
\end{align*}
Since
\[
\dbinom{m+j}{j}t^{j}=\frac{1}{m!}\frac{d^{m}}{dt^{m}}t^{m+j},
\]
we have%
\begin{align*}%
{\displaystyle\sum\limits_{n\geq k}}
\mathcal{R}_{n}^{k}\left(  m\right)  \frac{z^{n}}{n!}  &  =e^{z}\frac{m!}{k!}\left(
e^{-z}\frac{d}{dz}\right)  ^{m}\left(  \frac{1}{m!}e^{\left(  m-1\right)  z}%
{\displaystyle\sum\limits_{j=0}^{k}}
\left(  -1\right)  ^{k-j}\dbinom{k}{j}e^{jz}\right) \\
&  =e^{z}\left(  e^{-z}\frac{d}{dz}\right)  ^{m}\left(  \frac{1}{k!}e^{\left(
m-1\right)  z}\left(  e^{z}-1\right)  ^{k}\right)  .
\end{align*}
\end{proof}

Next, we obtain the following explicit relationship between the $m$-Stirling
numbers $\mathcal{R}_{n}^{k}\left(  m\right)  $ of the second kind and
$r$-Stirling numbers of the second kind.

\begin{theorem}
The following formula holds true%
\begin{equation}
\mathcal{R}_{n}^{k}\left(  m\right)  =%
{\displaystyle\sum\limits_{j=0}^{n}}
\dbinom{n}{j}\left(  1-m\right)  ^{n-j}%
{\displaystyle\sum\limits_{i=0}^{m}}
s\left(  m,i\right)  \mathcal{S}_{j+i}^{k}\left(  m-1\right)  . \label{F3}%
\end{equation}

\end{theorem}

\begin{proof}
Since
\[
\left(  e^{-z}\frac{d}{dz}\right)  ^{m}=e^{-mz}%
{\displaystyle\sum\limits_{i=0}^{m}}
s\left(  m,i\right)  \frac{d^{i}}{dz^{i}}%
\]
and
\[
\frac{1}{k!}e^{\left(  m-1\right)  z}\left(  e^{z}-1\right)  ^{k}=%
{\displaystyle\sum\limits_{n\geq k}}
\mathcal{S}_{n}^{k}\left(  m-1\right)  \frac{z^{n}}{n!},
\]
we can write (\ref{F2}) as
\begin{align*}%
{\displaystyle\sum\limits_{n\geq0}}
\mathcal{R}_{n}^{k}\left(  m\right)  \frac{z^{n}}{n!} &  =e^{\left(
1-m\right)  z}%
{\displaystyle\sum\limits_{i=0}^{m}}
s\left(  m,i\right)  \frac{d^{i}}{dz^{i}}\left(
{\displaystyle\sum\limits_{n\geq0}}
\mathcal{S}_{n}^{k}\left(  m-1\right)  \frac{z^{n}}{n!}\right)  \\
&  =e^{\left(  1-m\right)  z}%
{\displaystyle\sum\limits_{i=0}^{m}}
s\left(  m,i\right)
{\displaystyle\sum\limits_{n\geq0}}
\mathcal{S}_{n}^{k}\left(  m-1\right)  \frac{d^{i}}{dz^{i}}\frac{z^{n}}{n!}\\
&  =e^{\left(  1-m\right)  z}%
{\displaystyle\sum\limits_{i=0}^{m}}
s\left(  m,i\right)
{\displaystyle\sum\limits_{n\geq0}}
\mathcal{S}_{n}^{k}\left(  m-1\right)  \frac{z^{n-i}}{\left(  n-i\right)  !}\\
&  =e^{\left(  1-m\right)  z}%
{\displaystyle\sum\limits_{i=0}^{m}}
s\left(  m,i\right)
{\displaystyle\sum\limits_{l\geq0}}
\mathcal{S}_{l+i}^{k}\left(  m-1\right)  \frac{z^{l}}{l!}\\
&  =\left(
{\displaystyle\sum\limits_{n\geq0}}
\left(  1-m\right)  ^{n}\frac{z^{n}}{n!}\right)
{\displaystyle\sum\limits_{i=0}^{m}}
s\left(  m,i\right)  \left(
{\displaystyle\sum\limits_{n\geq0}}
\mathcal{S}_{n+i}^{k}\left(  m-1\right)  \frac{z^{n}}{n!}\right)  \\
&  =%
{\displaystyle\sum\limits_{n\geq0}}
\left(
{\displaystyle\sum\limits_{j=0}^{n}}
\dbinom{n}{j}\left(  1-m\right)  ^{n-j}%
{\displaystyle\sum\limits_{i=0}^{m}}
s\left(  m,i\right)  \mathcal{S}_{j+i}^{k}\left(  m-1\right)  \right)
\frac{z^{n}}{n!}.
\end{align*}
Now, by comparing the coefficients of $\frac{z^{n}}{n!}$ on both sides we
obtain $\left(  \ref{F3}\right)  $.
\end{proof}

\begin{theorem}
\label{EXX}The following explicit relationships hold true%
\begin{equation}
\mathcal{R}_{n}^{k}\left(  m\right)  =%
{\displaystyle\sum\limits_{i=0}^{m}}
\left(  -1\right)  ^{m-i}s\left(  m,i\right)  S(n+i,k). \label{F4}%
\end{equation}

\end{theorem}

\begin{proof}
We have%
\begin{align*}%
RHS  &  =%
{\displaystyle\sum\limits_{i=0}^{m}}
\left(  -1\right)  ^{m-i}s\left(  m,i\right)  \frac{1}{k!}%
{\displaystyle\sum\limits_{j=0}^{k}}
\left(  -1\right)  ^{k-j}\dbinom{k}{j}j^{n+i}\\
&  =\frac{1}{k!}%
{\displaystyle\sum\limits_{j=0}^{k}}
\left(
{\displaystyle\sum\limits_{i=0}^{m}}
\left(  -1\right)  ^{m-i}s\left(  m,i\right)  j^{i}\right)  \left(  -1\right)
^{k-j}\dbinom{k}{j}j^{n}\\
&  =\frac{1}{k!}%
{\displaystyle\sum\limits_{j=0}^{k}}
\left(  -1\right)  ^{m}\left(
{\displaystyle\sum\limits_{i=0}^{m}}
s\left(  m,i\right)  \left(  -j\right)  ^{i}\right)  \left(  -1\right)
^{k-j}\dbinom{k}{j}j^{n}\\
&  =\frac{1}{k!}%
{\displaystyle\sum\limits_{j=0}^{k}}
\left(  -1\right)  ^{m+k-j}\left(  -j\right)  \left(  -j-1\right)
\cdots(-j+m-1)\dbinom{k}{j}j^{n}\\
&  =\frac{1}{k!}%
{\displaystyle\sum\limits_{j=0}^{k}}
\left(  -1\right)  ^{k-j}\dbinom{k}{j}j^{n}\left(  j\right)  _{m}\\
&  =\mathcal{R}_{n}^{k}\left(  m\right)  .
\end{align*}

\end{proof}
\begin{remark}
By means of the formula (\ref{F4}), we can compute several values of
$\mathcal{R}_{n}^{k}\left(  m\right)  $ given by%
\begin{align*}
\mathcal{R}_{0}^{0}\left(  0\right)  &   =1,\\ \mathcal{R}_{0}^{0}\left(
m\right) & =0\text{ }(m\geq1),\\
\mathcal{R}_{0}^{k}\left(  m\right)   &  =%
\genfrac{\lfloor}{\rfloor}{0pt}{0}{m}{k}%
\; (k>0,m>0),
\\ \mathcal{R}_{n}^{1}\left(  m\right) & =m! \; (n>0,m\geq0),\\
\mathcal{R}_{n}^{k}\left(  0\right)    & =S(n,k),
\\ \mathcal{R}_{n}^{k}\left(
1\right) &  =S(n+1,k),\\
\mathcal{R}_{n}^{n+m}\left(  m\right)   & =1,\\
\mathcal{R}_{n}^{k}\left(  m\right)   & =0\text{ }(k>n+m),
\end{align*}
where $%
\genfrac{\lfloor}{\rfloor}{0pt}{0}{m}{k}%
$ denotes the Lah numbers \cite{Comtet} given by
\[%
\genfrac{\lfloor}{\rfloor}{0pt}{0}{m}{k}%
=\frac{m!}{k!}\dbinom{m-1}{k-1}.
\]
\end{remark}

\begin{theorem}
The $m-$Stirling numbers of the second kind $\mathcal{R}_{n}^{k}\left(
m\right)  $ satisfy the triangular recurrence relation%
\begin{equation}
\mathcal{R}_{n+1}^{k}\left(  m\right)  =\mathcal{R}_{n}^{k-1}\left(  m\right)
+k\mathcal{R}_{n}^{k}\left(  m\right)  \label{F5}%
\end{equation}
for $1\leq k\leq n+m$, with initial conditions:  \[
\mathcal{R}_{n}^{0}\left(
m\right)  =\delta_{0,n+m}\;\;(n,m\geq0) \] and
\[
\mathcal{R}_{0}^{k}\left(  m\right)  =%
\genfrac{\lfloor}{\rfloor}{0pt}{0}{m}{k}%
\; (k>0,m\geq0)
\]
 with $\delta_{i,j}$ being the Kronecker delta defined by%
\[
\delta_{i,j}=\left\{
\begin{array}
[c]{c}%
0\text{ \ }\left(  i\neq j\right)  ,\\
1\text{ \ }\left(  i=j\right)  .
\end{array}
\right.
\]
\end{theorem}
\begin{proof}
The result follows directly from the formula (\ref{F4}) and the recurrence formula for the Stirling numbers of the second kind.
\end{proof}
In the special case when $m=0,$ the triangular recurrence relation (\ref{F5})
corresponds to the well-known triangular recurrence for the Stirling numbers
$S(n,k)$ of the second kind. For $m=1,2,3$, we obtain the following tables, for $0\leq  n\leq 7$ and $0\leq k\leq m+7$.
\[%
\begin{tabular}
[c]{|cccccccccc|}\hline
\multicolumn{1}{|c|}{$n\backslash k$} & $0$ & $1$ & $2$ & $3$ & $4$ & $5$ &
$6$ & $7$ & $8$\\\hline
\multicolumn{1}{|c|}{$0$} & $0$ & $1$ &  &  &  &  &  &  & \\
\multicolumn{1}{|c|}{$1$} & $0$ & $1$ & $1$ &  &  &  &  &  & \\
\multicolumn{1}{|c|}{$2$} & $0$ & $1$ & $3$ & $1$ &  &  &  &  & \\
\multicolumn{1}{|c|}{$3$} & $0$ & $1$ & $7$ & $6$ & $1$ &  &  &  & \\
\multicolumn{1}{|c|}{$4$} & $0$ & $1$ & $15$ & $25$ & $10$ & $1$ &  &  & \\
\multicolumn{1}{|c|}{$5$} & $0$ & $1$ & $31$ & $90$ & $65$ & $15$ & $1$ &  &
\\
\multicolumn{1}{|c|}{$6$} & $0$ & $1$ & $63$ & $301$ & $350$ & $140$ & $21$ &
$1$ & \\
\multicolumn{1}{|c|}{$7$} & $0$ & $1$ & $127$ & $966$ & $1701$ & $1050$ &
$266$ & $28$ & $1$\\\hline
\end{tabular}
\
\]
\begin{center}
$\allowbreak \allowbreak $Table 1: Some values for the $1$-Stirling numbers of the second kind
\end{center}

\[%
\begin{tabular}
[c]{|ccccccccccc|}\hline
\multicolumn{1}{|c|}{$n\backslash k$} & $0$ & $1$ & $2$ & $3$ & $4$ & $5$ &
$6$ & $7$ & $8$ & $9$\\\hline
\multicolumn{1}{|c|}{$0$} & $0$ & $2$ & $1$ &  &  &  &  &  &  & \\
\multicolumn{1}{|c|}{$1$} & $0$ & $2$ & $4$ & $1$ &  &  &  &  &  & \\
\multicolumn{1}{|c|}{$2$} & $0$ & $2$ & $10$ & $7$ & $1$ &  &  &  &  & \\
\multicolumn{1}{|c|}{$3$} & $0$ & $2$ & $22$ & $31$ & $11$ & $1$ &  &  &  & \\
\multicolumn{1}{|c|}{$4$} & $0$ & $2$ & $46$ & $115$ & $75$ & $16$ & $1$ &  &
& \\
\multicolumn{1}{|c|}{$5$} & $0$ & $2$ & $94$ & $391$ & $415$ & $155$ & $22$ &
$1$ &  & \\
\multicolumn{1}{|c|}{$6$} & $0$ & $2$ & $190$ & $1267$ & $2051$ & $1190$ &
$287$ & $29$ & $1$ & \\
\multicolumn{1}{|c|}{$7$} & $0$ & $2$ & $382$ & $3991$ & $9471$ & $8001$ &
$2912$ & $490$ & $37$ & $1$\\\hline
\end{tabular}
\
\]
\begin{center}
$\allowbreak \allowbreak $Table 2: Some values for the $2$-Stirling numbers of the second kind
\end{center}

\[%
\begin{tabular}
[c]{|c|ccccccccccc|}\hline
\multicolumn{1}{|c|}{$n\backslash k$} & $0$ & $1$ & $2$ & $3$ & $4$ & $5$ &
$6$ & $7$ & $8$ & $9$ & $10$\\\hline
\multicolumn{1}{|c|}{$0$} & $0$ & $6$ & $6$ & $1$ &  &  &  &  &  &  & \\
\multicolumn{1}{|c|}{$1$} & $0$ & $6$ & $18$ & $9$ & $1$ &  &  &  &  &  & \\
\multicolumn{1}{|c|}{$2$} & $0$ & $6$ & $42$ & $45$ & $13$ & $1$ &  &  &  &  &
\\
\multicolumn{1}{|c|}{$3$} & $0$ & $6$ & $90$ & $177$ & $97$ & $18$ & $1$ &  &
&  & \\
\multicolumn{1}{|c|}{$4$} & $0$ & $6$ & $186$ & $621$ & $565$ & $187$ & $24$ &
$1$ &  &  & \\
\multicolumn{1}{|c|}{$5$} & $0$ & $6$ & $378$ & $2049$ & $2881$ & $1550$ &
$331$ & $31$ & $1$ &  & \\
\multicolumn{1}{|c|}{$6$} & $0$ & $6$ & $762$ & $6525$ & $13573$ & $10381$ &
$3486$ & $548$ & $39$ & $1$ & \\
\multicolumn{1}{|c|}{$7$} & $0$ & $6$ & $1530$ & $20337$ & $60817$ & $65478$ &
$31297$ & $7322$ & $860$ & $48$ & $1$\\\hline
\end{tabular}
\
\]
\begin{center}
$\allowbreak \allowbreak $Table 3: Some values for the $3$-Stirling numbers of the second kind
\end{center}

Now, we define the weighted $m$-Stirling numbers of the second kind
$\mathcal{R}_{n}^{k}\left(  x;m\right)  $ by%
\begin{equation}
\mathcal{R}_{n}^{k}\left(  x;m\right)  =\frac{1}{k!}%
{\displaystyle\sum\limits_{j=0}^{k}}
\left(  -1\right)  ^{k-j}\dbinom{k}{j}\left(  j+x\right)  ^{n}\left(
j\right)  _{m}. \label{xstirling}%
\end{equation}
From the definition of $\mathcal{R}_{n}^{k}\left(  x;m\right)  $ we can derive
a recurrence formula%
\[
\mathcal{R}_{n+1}^{k}\left(  x;m\right)  =\mathcal{R}_{n}^{k-1}\left(
x;m\right)  +\left(  x+k\right)  \mathcal{R}_{n}^{k}\left(  x;m\right)  .
\]
Setting $x=0$ in \eqref{xstirling} yields formula \eqref{ist}.%

In the next paragraph, we describe a connection between the $m$-Stirling
numbers of the second kind and graph theory.
\begin{remark}
Let $G$ be a finite and simple graph with $n$ vertices and
$P$ its chromatic polynomial. Stanley in \cite[Theorem 1.2]{Stan06} proved for
all non-negative integer $x$ the following
\[
\overline{P}(x)=(-1)^{n}P(-x)
\]
where $\overline{P}(x)$ is the number of pairs ($\sigma$, $\mathcal{O}$), with
$\sigma$ is any map $\sigma:V\rightarrow\{1,2,\dots,x\}$ and $\mathcal{O}$ is
an orientation of $G$. We say that $\sigma$ is compatible with $\mathcal{O}$
if the following conditions are satisfied

\begin{enumerate}
\item The orientation $\mathcal{O}$ is acyclic,

\item If $u\rightarrow v$ in the orientation $\mathcal{O}$, then
$\sigma(u)\geq\sigma(v)$.
\end{enumerate}

Moreover, we have the following properties of $\overline{P}$

\begin{enumerate}
\item $\overline{P}(G_{0},x)=x$, where $G_{0}$ is the one-vertex graph,

\item $\overline{P}(G+H,x)=\overline{P}(G,x)\overline{P}(H,x)$, where $G+H$ denotes the disjoint union of graphs $G$ and $H$,

\item $\overline{P}(G,x)=\overline{P}(G-e,x)+\overline{P}(G / e,x)$.
\end{enumerate}
Where $G-e$ and $G/e$ are graphs obtained from $G$ by respectively
deleting and contracting an edge $e$.
\end{remark}

\begin{theorem}
We have%
\[
\mathcal{R}_{n}^{k}(m)=\displaystyle\frac{1}{k!}\sum_{j=0}^{k}(-1)^{k-j}%
{\binom{k}{j}}\overline{P}(O_{n}+K_{m},j).
\]
\end{theorem}
\begin{proof}
It is known from \cite{Dong} that
\[
P(O_{n},x)=x^{n},
\]
and%

\[
P(K_{n},x)=(-1)^{n}(-x)_{n},
\]
where $O_{n}$ and $K_{n}$ are the empty graph and the complete graph
respectively. Then by making use of $(-1)^{n}(-x)_{n}=x(x-1)\dots(x-n+1)$, we
deduce that%
\[
\overline{P}(O_{n},x)=x^{n},
\]
and
\begin{align}
\overline{P}(K_{n},x)  &  =(-1)^{n}P(K_{n},-x)\nonumber\\
&  =(x)_{n}\nonumber
\end{align}
Finally, from \eqref{Amn} we get the desired result.
\end{proof}
\section{The $m-$Hurwitz type poly-Bernoulli numbers with negative upper indices}

In this section, we consider an explicit formula for $m-$Hurwitz type
poly-Bernoulli numbers $\mathbb{B}_{n,m}^{(-k)}\left(  a\right)  $ with
negative upper indices involving $m-$Stirling numbers of the second kind
$\mathcal{R}_{n}^{k}\left(  m\right)  .$

\begin{theorem}\label{ist1}
We have
\[
\mathbb{B}_{n,m}^{(-k)}\left(  a\right)  =\frac{a^{k}}{m!\left(  m+a\right)
^{k}}%
{\displaystyle\sum\limits_{l=0}^{\min(n,k)}}
\left(  l!\right)  ^{2}\mathcal{S}_{k}^{l}\left(  a\right)  \mathcal{R}_{n+1}^{l+1}(m).
\]
\end{theorem}
\begin{proof}
From (\ref{Form2}) and theorem $2.2$ on page $804$ of \cite{Cencki}, we get%
\begin{align*}
\mathbb{B}_{n,m}^{(-k)}\left(  a\right)   &  =\frac{\left(  -1\right)
^{m+n}a^{k}}{m!\left(  m+a\right)  ^{k}}%
{\displaystyle\sum\limits_{i=0}^{m}}
\left(  -1\right)  ^{n+i}s\left(  m,i\right)  \mathcal{HB}_{n+i}^{\left(  -k\right)
}\left(  a\right)  \\
&  =\frac{\left(  -1\right)  ^{m}a^{k}}{m!\left(  m+a\right)  ^{k}}%
{\displaystyle\sum\limits_{l=0}^{{}}}
\left(  l!\right)  ^{2}\mathcal{S}_{k}^{l}\left(  a\right)
{\displaystyle\sum\limits_{i=0}^{m}}
\left(  -1\right)  ^{i}s\left(  m,i\right)
\mathcal{S}({n+i+1},{l+1})%
.
\end{align*}
The conclusion follows by Theorem \ref{EXX}.
\end{proof}

As a consequence of Theorem \ref{TTT}, one can deduce a three-term recurrence
relation for $\mathbb{B}_{n,m}^{(-k)}\left(  a\right)  $.

\begin{corollary}
The $\mathbb{B}_{n,m}^{(-k)}\left(  a\right)  $ satisfies the following
three-term recurrence relation:%
\begin{equation}
\mathbb{B}_{n+1,m}^{(-k)}\left(  a\right)  =\frac{\left(  m+1\right)  \left(
m+a+1\right)  ^{k}}{\left(  m+a\right)  ^{k}}\mathbb{B}_{n,m+1}^{(-k)}\left(
a\right)  -m\mathbb{B}_{n,m}^{(-k)}\left(  a\right)  ,
\end{equation}
with the initial sequence given by
\[
\mathbb{B}_{0,m}^{\left(  -k\right)  }\left(  a\right)  =a^{k}.
\]

\end{corollary}
If $m=0$ and $a=1$, then Theorem \ref{ist1} reduces to the duality property of poly-Bernoulli numbers \cite{Kaneko1}.
\begin{corollary}
We have
\[
\mathcal{B}_{n}^{\left(  -k\right)  }=\mathcal{B}_{k}^{\left(  -n\right)  }.
\]

\end{corollary}
\section{The $m$-Hurwitz type poly-Bernoulli polynomials}

For $m\geq0,$ let us consider the $m$-Hurwitz type poly-Bernoulli
polynomials $\mathbb{B}_{n,m}^{\left(  k\right)  }\left(  x;a\right)  $ as
follows:
\begin{equation}
\mathbb{B}_{n,m}^{\left(  k\right)  }\left(  x;a\right)  =%
{\displaystyle\sum\limits_{i=0}^{n}}
\left(  -1\right)  ^{n-i}\dbinom{n}{i}\mathbb{B}_{i,m}^{(k)}\left(  a\right)
x^{n-i}. \label{SSD}%
\end{equation}
It is easy to show that the generating function of $\mathbb{B}_{n,m}^{\left(
k\right)  }\left(  x;a\right)  $ is given by%
\begin{multline*}%
{\displaystyle\sum\limits_{n\geq0}}
\mathbb{B}_{n,m}^{\left(  k\right)  }\left(  x;a\right)  \frac{z^{n}}%
{n!}=e^{-xz}{\displaystyle\sum\limits_{n\geq0}}\mathbb{B}_{n,m}^{(k)}\left(
a\right)  \frac{z^{n}}{n!}\\
=\frac{1}{m!}e^{-\left(  m+x\right)  z}\left(  1+\frac{m}{a}\right)
^{k}\left(  e^{z}\frac{d}{dz}\right)  ^{m} \left(
\left(  1-e^{-z}\right)  ^{m}\Phi\left(  1-e^{-z},k,m+a\right)\right)  .
\end{multline*}
Next, we state an explicit formula for $\mathbb{B}_{n,m}^{\left(  k\right)
}\left(  x;a\right)  $.

\begin{theorem}
The following formula holds true%
\[
\mathbb{B}_{n,m}^{\left(  k\right)  }\left(  x;a\right)  =\frac{\left(
m+a\right)  ^{k}}{m!a^{k}}\sum_{i=0}^{n}\frac{\left(  -1\right)  ^{n-i}\left(
i+m\right)  !}{\left(  i+m+a\right)  ^{k}}\mathcal{S}_{n}^{i}\left(
x+m\right)  .
\]

\end{theorem}

\begin{proof}
We have
\begin{align*}
&
{\displaystyle\sum\limits_{n\geq0}}
\left(  \frac{\left(  m+a\right)  ^{k}}{m!a^{k}}\sum_{i=0}^{n}\frac{\left(
-1\right)  ^{n-i}\left(  i+m\right)  !}{\left(  i+m+a\right)  ^{k}}%
\mathcal{S}_{n}^{i}\left(  x+m\right)  \right)  \frac{z^{n}}{n!}\\
&  =\frac{\left(  m+a\right)  ^{k}}{m!a^{k}}\sum_{i\geq0}\frac{\left(
-1\right)  ^{i}\left(  i+m\right)  !}{\left(  i+m+a\right)  ^{k}}%
{\displaystyle\sum\limits_{n\geq i}}
\mathcal{S}_{n}^{i}\left(  x+m\right)  \frac{\left(  -z\right)  ^{n}}{n!}\\
&  =\frac{\left(  m+a\right)  ^{k}}{m!a^{k}}\sum_{i\geq0}\frac{\left(
-1\right)  ^{i}\left(  i+m\right)  !}{\left(  i+m+a\right)  ^{k}}\frac{1}%
{i!}e^{-\left(  x+m\right)  z}\left(  e^{-z}-1\right)  ^{i}\\
&  =e^{-\left(  x+m\right)  z}\frac{\left(  m+a\right)  ^{k}}{a^{k}}%
\sum_{i\geq0}\dbinom{m+i}{i}\frac{\left(  1-e^{-z}\right)  ^{i}}{\left(
i+m+a\right)  ^{k}}\\
&  =e^{-xz}%
{\displaystyle\sum\limits_{n\geq0}}
\mathbb{B}_{n,m}^{(k)}\left(  a\right)  \frac{z^{n}}{n!}\\
&  =%
{\displaystyle\sum\limits_{n\geq0}}
\mathbb{B}_{n,m}^{\left(  k\right)  }\left(  x;a\right)  \frac{z^{n}}{n!}.
\end{align*}
Therefore, we get the desired result by comparing the coefficients of
$\frac{z^{n}}{n!}$ on both sides.
\end{proof}

\begin{theorem}
The polynomials $\mathbb{B}_{n,m}^{\left(  k\right)  }\left(  x;a\right)  $
satisfy the following three-term recurrence relation:%
\begin{equation}
\mathbb{B}_{n+1,m}^{\left(  k\right)  }\left(  x;a\right)  =\frac{\left(
m+1\right)  \left(  m+a\right)  ^{k}}{\left(  m+a+1\right)  ^{k}}%
\mathbb{B}_{n,m+1}^{\left(  k\right)  }\left(  x;a\right)  +\left(
x-m\right)  \mathbb{B}_{n,m}^{\left(  k\right)  }\left(  x;a\right)
\label{Recx}%
\end{equation}
with the initial sequence given by
\[
\mathbb{B}_{n,0}^{\left(  k\right)  }\left(  x;a\right)  =\frac{1}{a^{k}}.
\]

\end{theorem}

\begin{proof}
From (\ref{SSD}), we get
\begin{align*}
x\frac{d}{dx}\mathbb{B}_{n,m}^{\left(  k\right)  }\left(  x;a\right)   &  =n%
{\displaystyle\sum\limits_{j=0}^{n}}
\left(  -1\right)  ^{n-j}\binom{n}{j}\mathbb{B}_{j,m}^{\left(  k\right)
}\left(  a\right)  x^{n-j}-n%
{\displaystyle\sum\limits_{j=1}^{n}}
\left(  -1\right)  ^{n-j}\binom{n-1}{j-1}\mathbb{B}_{j,m}^{\left(  k\right)
}\left(  a\right)  x^{n-j}\\
&  =n\mathbb{B}_{n,m}^{\left(  k\right)  }\left(  x;a\right)  -n%
{\displaystyle\sum\limits_{j=0}^{n-1}}
\left(  -1\right)  ^{n-j-1}\binom{n-1}{j}\mathbb{B}_{j+1,m}^{\left(  k\right)
}\left(  a\right)  x^{n-j-1}.
\end{align*}
Now, using (\ref{Mach}), we have%
\begin{multline*}
x\frac{d}{dx}\mathbb{B}_{n,m}^{\left(  k\right)  }\left(  x;a\right)
=n\mathbb{B}_{n,m}^{\left(  k\right)  }\left(  x;a\right)  +nm%
{\displaystyle\sum\limits_{j=0}^{n-1}}
\left(  -1\right)  ^{n-j-1}\binom{n-1}{j}\mathbb{B}_{j,m}^{\left(  k\right)
}\left(  a\right)  x^{n-j-1}\\
-n\frac{\left(  m+1\right)  \left(  m+a\right)  ^{k}}{\left(  m+a+1\right)
^{k}}%
{\displaystyle\sum\limits_{j=0}^{n-1}}
\left(  -1\right)  ^{n-j-1}\binom{n-1}{j}\mathbb{B}_{j,m+1}^{\left(  k\right)
}\left(  a\right)  x^{n-j-1}%
\end{multline*}
which, after simplification, yields%
\[
x\mathbb{B}_{n-1,m}^{\left(  k\right)  }\left(  x;a\right)  =\mathbb{B}%
_{n,m}^{\left(  k\right)  }\left(  x;a\right)  -\frac{\left(  m+1\right)
\left(  m+a\right)  ^{k}}{\left(  m+a+1\right)  ^{k}}\mathbb{B}_{n-1,m+1}%
^{\left(  k\right)  }\left(  x;a\right)  +m\mathbb{B}_{n-1,m}^{\left(
k\right)  }\left(  x;a\right)  ,
\]
which is obviously equivalent to (\ref{Recx}) and the proof is complete.
\end{proof}

The next lemma is used in the proof of the Theorem \ref{Dual}.

\begin{lemma}
\label{Lemm}We have%
\[%
{\displaystyle\sum\limits_{i=0}^{n}}
\left(  -1\right)  ^{n-i}\dbinom{n}{i}\mathcal{R}_{i+1}^{l+1}(m)x^{n-i}%
=\mathcal{R}_{n+1}^{l+1}(-x;m)+x\mathcal{R}_{n}^{l+1}(-x;m)
\]
and for $m=0,$ we have%
\[%
{\displaystyle\sum\limits_{i=0}^{n}}
\left(  -1\right)  ^{n-i}\dbinom{n}{i}S(i+1,l+1)x^{n-i}=\mathcal{S}_{n}%
^{l}\left(  1-x\right).
\]

\end{lemma}

\begin{proof} We have
\begin{align*}
&
{\displaystyle\sum\limits_{i=0}^{n}}
\left(  -1\right)  ^{n-i}\dbinom{n}{i}\mathcal{R}_{i+1}^{l+1}(m)x^{n-i}\\
& =%
{\displaystyle\sum\limits_{i=0}^{n}}
\left(  -1\right)  ^{n-i}\dbinom{n}{i}x^{n-i}\left(  \frac{1}{\left(
l+1\right)  !}\sum_{j=0}^{l+1}(-1)^{l+1-j}{\binom{l+1}{j}}j^{i+1}%
(j)_{m}\right)  \\
& =\frac{1}{\left(  l+1\right)  !}\sum_{j=0}^{l+1}(-1)^{l+1-j}{\binom{l+1}{j}%
}(j)_{m}%
{\displaystyle\sum\limits_{i=0}^{n}}
\left(  -1\right)  ^{n-i}\dbinom{n}{i}j^{i+1}x^{n-i}\\
& =\frac{1}{\left(  l+1\right)  !}\sum_{j=0}^{l+1}(-1)^{l+1-j}{\binom{l+1}{j}%
}(j)_{m}j\left(  j-x\right)  ^{n}\\
& =\frac{1}{\left(  l+1\right)  !}\sum_{j=0}^{l+1}(-1)^{l+1-j}{\binom{l+1}{j}%
}(j)_{m}\left(  \left(  j-x\right)  +x\right)  \left(  j-x\right)  ^{n}.
\end{align*}
This completes the proof of lemma.
\end{proof}

\begin{theorem}
\label{Dual}We have%
\[
\mathbb{B}_{n,m}^{\left(  -k\right)  }\left(  x;a\right)  =\frac{a^{k}%
}{m!\left(  m+a\right)  ^{k}}%
{\displaystyle\sum\limits_{l=0}^{\min(n,k)}}
\left(  l!\right)  ^{2}\mathcal{S}_{k}^{l}\left(  a\right)  \left(  \mathcal{R}%
_{n+1}^{l+1}(-x;m)+x\mathcal{R}_{n}^{l+1}(-x;m)\right)  .
\]

\end{theorem}

\begin{proof}
We have%
\begin{align*}
\mathbb{B}_{n,m}^{\left(  -k\right)  }\left(  x;a\right)   &  =%
{\displaystyle\sum\limits_{i=0}^{n}}
\left(  -1\right)  ^{n-i}\dbinom{n}{i}\mathbb{B}_{i,m}^{(-k)}\left(  a\right)
x^{n-i}\\
&  =%
{\displaystyle\sum\limits_{i=0}^{n}}
\left(  -1\right)  ^{n-i}\dbinom{n}{i}\left(  \frac{a^{k}}{m!\left(
m+a\right)  ^{k}}%
{\displaystyle\sum\limits_{l=0}^{\min(n,k)}}
\left(  l!\right)  ^{2}\mathcal{S}_{k}^{l}\left(  a\right)  \mathcal{R}_{i+1}%
^{l+1}(m).\right)  x^{n-i}\\
&  =\frac{a^{k}}{m!\left(  m+a\right)  ^{k}}%
{\displaystyle\sum\limits_{l=0}^{\min(n,k)}}
\left(  l!\right)  ^{2}\mathcal{S}_{k}^{l}\left(  a\right)
{\displaystyle\sum\limits_{i=0}^{n}}
\left(  -1\right)  ^{n-i}\dbinom{n}{i}\mathcal{R}_{i+1}^{l+1}(m)x^{n-i}.
\end{align*}
The result follows from the Lemma \ref{Lemm}.
\end{proof}
\begin{corollary}
We have
\begin{equation}
\mathbb{B}_{n}^{\left(  -k\right)  }\left(  x\right)  =%
{\displaystyle\sum\limits_{l=0}^{\min(n,k)}}
\left(  l!\right)  ^{2}%
\mathcal{S}({k+1},{l+1})%
\mathcal{S}_{n}^{l}\left(  1-x\right)  .\label{last}
\end{equation}
\end{corollary}
\begin{proof}
The proof follows from Theorem \ref{Dual} with $m=0$ and $a=1$.
\end{proof}
Note that the formula \eqref{last} was already given in \cite{Cencki}.


\begin{thebibliography}{99}

\bibitem{Kaneko1} T. Arakawa, T. Ibukiyama and M. Kaneko, {\it Bernoulli numbers and zeta functions. With an appendix by Don Zagier}, Springer, Tokyo, (2014).
\bibitem{Broder} A.Z. Broder, {\it The $r$-Stirling numbers}, Discrete Math. 49 (1984), pp. 241--259.
\bibitem{Carlitz1} L. Carlitz, {\it Weighted Stirling numbers of the first and second kind. I}, Fibonacci Q. 18 (1980), pp. 147--162.
\bibitem{Carlitz2} L. Carlitz, {\it Weighted Stirling numbers of the first and second kind. II}, Fibonacci Q. 18 (1980), pp. 242--257.
\bibitem{Cencki} M. Cenkci and P.T. Young, {\it Generalizations of poly-Bernoulli and poly-Cauchy numbers}, Eur. J. Math. 1 (2015), pp. 799--828 .
\bibitem{Comtet} L. Comtet, {\it Advanced combinatorics. The art of finite and infinite expansions}, D. Reidel Publishing Company. X, Dordrecht, Holland - Boston, U.S.A., (1974).
\bibitem{Dong} F.M. Dong, K.M. Koh and K.L. Teo, {\it Chromatic polynomials and chromaticity of graphs}, World Scientific Publishing, Singapore, (2005).
\bibitem{Kaneko2} M. Kaneko, {\it Poly-{B}ernoulli numbers}, J. Th\'eor. Nombres Bordeaux 9 (1997), pp. 221--228.
\bibitem{Rahmani2014} M. Rahmani, {\it Generalized Stirling transform}, Miskolc Math. Notes 15 (2014), pp. 677--690.
\bibitem{Rahmani3} M. Rahmani, {\it On $p$-Bernoulli numbers and polynomials}, J. Number Theory 157 (2015), pp. 350--366.
\bibitem{Srivastava2} H.M. Srivastava and J. Choi, {\it Zeta and $q$-zeta functions and associated series and integrals}, Elsevier, Amsterdam, (2012).
\bibitem{Stan06} R.P. Stanley, {\it Acyclic orientations of graphs }, Discrete Math. 306 (2006), pp. 905--909.

\end{thebibliography}
\end{document}